\documentclass[12pt]{article}

\usepackage{tikz}

\usepackage{amsmath}
\usepackage{amsfonts}
\usepackage{amssymb}

\newtheorem{definition}{Definition}

\newtheorem{fig}{Figure}

\newtheorem{theorem}{Theorem}
\newtheorem{lemma}{Lemma}
\newtheorem{proposition}{Proposition}
\newtheorem{corollary}{Corollary}

\newenvironment{proof}{\textbf{Proof} :}{$\square$}

\title{A Sufficient Condition for J\'onsson's Conjecture and its Relationship with Finite Semidistributive lattices}
\author{Brian T. Chan}
\date{University of Calgary}

\begin{document}

\maketitle

\begin{abstract}
This article is part of my upcoming masters thesis which investigates the following open problem from the book, Free Lattices, by R.Freese, J.Jezek, and J.B. Nation published in 1995 \cite{FL}: ``Which lattices (and in particular which countable lattices) are sublattices of a free lattice?'' \\

Despite partial progress over the decades, the problem is still unsolved. There is emphasis on the countable case because the current body of knowledge on sublattices of free lattices is most concentrated on when these sublattices are countably infinite. \\

It is known that sublattices of free lattices which are \emph{finite} can be characterized as being those lattices which satisfy Whitman's condition and are semidistributive. This assertion was conjectured by B. J\'onsson in the 1960's and proven by J.B. Nation in 1980. However, there is a desire for a new proof to this deep result (see \cite{FL}) as Nation's proof is very involved and more insight into sublattices of free lattices is sought after. \\

In this article, a sufficient condition involving a construct known as a \emph{join minimal pair}, or just a minimal pair, implying J'onsson's conjecture is proven. Minimal pairs were first defined by H. Gaskill in \cite{SRTL} when analysing \emph{sharply transferable lattices}. Using this sufficient condition, research by I.Rival and B.Sands (\cite{PSL} and \cite{PSL2}) is used to compare this condition with properties of finite semidistributive lattices and in the process refute the main assertion of a manuscript by H.M\"uhle (\cite{HM}). Moreover, inspired by the approaches used by Henri M\"uhle in \cite{HM}, I will make a partial result which investigates a possible forbidden sublattice characterization involving breadth-two planar semidistributive lattices. To the best of my knowledge, the two results of this article (the sufficient condition for J\'onsson's conjecture and the partial result aforementioned) are new. \\
\end{abstract}

\section{The $\kappa$ function on finite semidistributive lattices}

On finite semidistributive lattices there is a special bijective function in the literature, denoted by $\kappa$ in \cite{FL}, that matches join irreducible elements with meet irreducible elements and vice versa. Roughly speaking, this functions describes the structural symmetry a finite semidistributive lattice has (interestingly enough, this function is not even an order preserving map).\\

This function can be found in both \cite{FL} and \cite{VL}. We will make explicit the relationship, described in \cite{FL}, between join irreducible elements and the image of these elements under $\kappa$. \\

Let $L$ be a finite semidistributive lattice, and $x \in J(L)$. If $y,z \in L$ satisfy $y \geq x_*$, $z \leq x_*$, $y \ngeq x$, and $z \nleq x$ then we have $y+z \geq x_*$ and $y+z \ngeq x$. This is because the above is equivalent to saying: $yx = x_*$ and $zx = x_*$ implies $(y+z)x = x_*$. \\

Due to this we can define $\kappa(x) = \bigvee \{y \in L : yx = x_* \}$. This element, $\kappa(x)$, is the largest element $s$ of $L$ satisfying $s \geq x_*$ and $s \ngeq x$. Moreover, $\kappa(x)$ can be seen to be meet irreducible: if $\kappa(x) = ab$ for some $a,b \in$ $\uparrow x$ then $\kappa(x) = ab \in$ $\uparrow x$, contrary to the definition of $\kappa(x)$. Consider the following figure, this being Figure 2.5 from \cite{FL}. We note that $[a_*, \kappa(a)^*] = [a_*, \kappa(a)] \cup [a, \kappa(a)^*]$:

\begin{fig} \label{kappa}
	
	\begin{center}
		\begin{tikzpicture} [scale=1.5]
		
		\node(1) at (1,1) {$\kappa(a)^*$};
		
		\node(a) at (-1,0) {$a$};
		\node(b) at (1,0) {$\kappa(a)$};
		
		\node(0) at (-1,-1) {$a_*$};
		
		\draw (a) -- (0);
		\draw (b) -- (1);
		
		\draw (a) -- (1);
		\draw (0) -- (b);
		
		\end{tikzpicture}
	\end{center}
	
\end{fig}

It suggests that there is a one-to-one correspondence between join irreducible elements and meet irreducible elements. To see this, we note that $a$ is the smallest element of $\{s \in L : s \kappa(a) = \kappa(a)^*\}$, using the same arguments shown above. Hence, $\kappa : J(L) \to M(L)$ is a bijection. We also note, by virtue of the covering relation, that $a + \kappa(a) = \kappa(a)^*$ and $a \kappa(a) = a_*$. \\

The converse, that for finite lattices the existence of such a bijection, $\kappa : J(L) \to M(L)$, as described above implies semidistributivity, is also true and can be seen as follows. We use the short argument found in \cite{FL}, see Theorem 2.56 of that text. Without loss of generality, suppose that for some $a,b,c \in L$, $x = ab = ac$ but $x < a(b+c)$. Then $\kappa(a(b+c))$ does not exist since if it did, $\kappa(a(b+c)) \geq b$ and $\kappa(a(b+c)) \geq c$. \\

We now present this nice duality result from J.B. Nation.

\begin{lemma} [J.B. Nation. \cite{FL}, finite case of lemma 2.63] \label{duality}
	If $L$ is finite and semidistributive and $\kappa: J(L) \to M(L)$ is its bijection (as described above) then for all $a,b \in J(L)$: $a A b$ iff $\kappa(a) B^d \kappa(b)$, and $a B b$ iff $\kappa(a) A^d \kappa(b)$. The relations $A^d$ and $B^d$ are the dual definitions of $A$ and $B$ respectively on $M(L)$.
\end{lemma}
\begin{proof}[This is from \cite{FL}]
	Using the definition of $\kappa$, we show that $a A b$ iff $b < a \leq \kappa(b)^*$ and that $a B b$ iff $b_* \leq \kappa(a) < \kappa(b)$. When $a A b$, $b < a$ and $a \leq b \vee \kappa(b) = \kappa(b)^*$. Hence $a A b$ iff $b < a \leq \kappa(b)^*$. With $a B b$, $a_* \vee b_* \ngeq a$ iff $b_* \leq \kappa(a)$ and $a_* \vee b \geq a$ iff $b \nleq \kappa(a)$. So we have $b_* \leq \kappa(a) \leq \kappa(b)$. Because $\kappa$ is a bijection, we conclude that $a B b$ iff $b_* \leq \kappa(a) < \kappa(b)$.
\end{proof}

From the above proof we see :

\begin{corollary}[\cite{FL}] \label{sym form}
	Let $L$ be a finite semidistributive lattice and $a,b \in J(L)$. Then $a A b$ iff $b < a \leq \kappa(b)^*$ and $a B b$ iff $b_* \leq \kappa(a) < \kappa(b)$.
\end{corollary}

We first describe what (join) minimal pairs are:

\begin{definition}{Gaskill, Gr\"atzer, and Platt \cite{STL}}

Let $L$ be a finite lattice, $a \in L$ and $B \subseteq L$. Then $\langle a ; B \rangle$ is a (join) minimal pair if $a \leq \bigvee B$ and for any $B' \subseteq L$ satisfying $a \leq \bigvee B'$, if for all $x \in B'$ there exists a $y \in B$ such that $x \leq y$ then $B \subseteq B'$

\end{definition}

\section{C-cycles viewed as Cycles of minimal pairs}

An application of the $\kappa$ function on finite semidistributive lattices is related to J\'onsson's conjecture. Specifically when using an important tool used by Nation and others known as a $C-$cycle (see \cite{FL}). J'onsson's conjecture, which was proven by J.B.Nation in 1980, states that every finite lattice that is semidistributive and satisfies Whitman's condition is a sublattice of a free lattice. The application is a small result that I made, and it does not seem to appear explicitly elsewhere in the literature. By viewing the main tool used by Nation to prove J\'onsson's conjecture using cycles of minimal pairs, as emphasized in \cite{STL}, I will make a small observation that gives a subtly different way of approaching Jonsson's conjecture. \\

The small observation is theorem \ref{fruitful}, involving cycles of minimal pairs. For brevity, we adopt the convention from \cite {FL} and call an \emph{S-lattice} a finite semidistributive lattice that satisfies Whitman's condition. We do not plan to go into the deep machinery of Nation's proof of J\'onsson's conjecture, one can consult \cite{SFL} or \cite{FL}; both have Nation's proof. Instead, a way of interpreting some notions surrounding J\'onsson's conjecture in a new way will be described. To the best of our knowledge, this interpretation is new. One feature that seems to pervade in the interpretation given here and in Nation's proof is the duality result, lemma \ref{duality}. Roughly summarizing, the duality result describes how $C-$cycles respect a structural symmetry, described by the $\kappa$ function, of finite semidistributive lattices. \\

Here is a simplified summary. How this duality result is used in Nation's proof will be described, from an overview in \cite{FL} of Nation's proof. In constructing a contradiction, Nation supposed that in some $S-$lattice there was a minimal $C-$cycle, where if any term in that cycle was removed, the remaining terms would no longer form a cycle. He made a delicate argument and combined it with the duality result, lemma \ref{duality}, to show the following. The only way to avoid a cycle of inequalities ${p_0}_{**} \geq {p_1}_{**} \geq \dots \geq {p_n}_{**} \geq {p_0}_{**}$, where $p_k$ is a join irreducible element and one of those inequalities is strict, is to assume that the cycle is of the form $p_0 B p_1 A p_2 B p_3 \dots A p_n = p_0$. Using separate arguments, he proved that this is also impossible. \\

We note that the below interpretation is inspired by Rival and Sands' 1978 paper, \cite{PSL}, where they proved J\'onsson's conjecture for finite lattices that are planar. A finite lattice is \emph{planar} if its Hasse diagram can be drawn in $\mathbb{R}^2$ in such a way that no two of its drawn line segments intersect. They use machinery from Gaskill and others, in their study of sharply transferable lattices. In particular, the goal was to show that cycles of minimal pairs do exist. Planarity seems to eliminate a lot of unwanted cases, as it allowed them to show that if such cycles exist, the minimal pairs have to be of the form $\langle p ; \{x, y\} \rangle$ (Figure 3 (a) of \cite{PSL}. In the below they allow for the possibility that $p = x \vee y$. \\

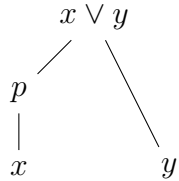
\begin{center}
	\begin{tikzpicture} [scale=1]
	
	\node (join) at (0,2) {$x \vee y$};
	
	\node (p) at (-1,1) {$p$};
	\node (x) at (-1,0) {$x$};
	
	\node (y) at (1,0) {$y$};
	
	\draw (x) -- (p) -- (join) -- (y);
	
	\end{tikzpicture}
\end{center}

Such a claim could be made since the $S-$lattices they were looking at were, as they were planar, simplified enough. In fact such a simplification, as described in \cite{PSL}, could also be made with $S-$lattices with a breadth of two. We first describe $C-$cycles in a slightly different way using minimal pairs. This can be deduced from Gaskill, Gr\"atzer, and Platt's perspective of J\'onsson's conjecture (see \cite{STL}) or even by viewing $C-$cycles as a special form of the join-dependency relation. \\

Consider a finite semidistributive lattice $L$ with associated function $\kappa$. In particular, assume that the $p_k$ are join irreducible. Have the left diagram denote $p_k A p_{k+1}$ and the right diagram denote $p_k B p_{k+1}$:

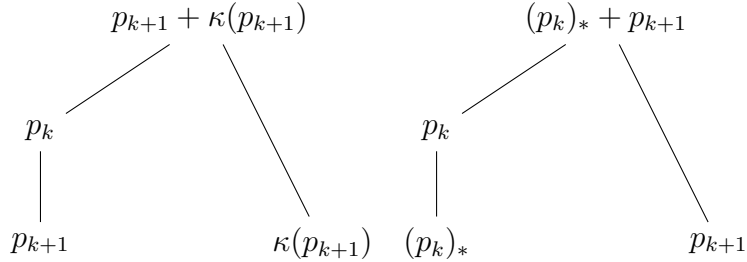
\begin{center}
	\begin{tikzpicture} [scale=1.5]
	
	\node (1) at (0,1) {$p_{k+1} + \kappa(p_{k+1})$};
	\node (a) at (-1.5,0) {$p_k$};
	\node (b) at (-1.5,-1) {$p_{k+1}$};
	\node (x) at (1,-1) {$\kappa(p_{k+1})$};
	
	\draw (b) -- (a) -- (1) -- (x);
	
	\end{tikzpicture}
	\begin{tikzpicture} [scale=1.5]
	
	\node (1) at (0,1) {$(p_k)_* + p_{k+1}$};
	\node (a) at (-1.5,0) {$p_k$};
	\node (x) at (-1.5,-1) {$(p_k)_*$};
	\node (b) at (1,-1) {$p_{k+1}$};
	
	\draw (x) -- (a) -- (1) -- (b);
	
	\end{tikzpicture}
\end{center}

In the above left diagram, $p_k \nleq \kappa(p_{k+1})$ since $p_{k+1} \nleq \kappa(p_{k+1})$. In the above right diagram, $p_k \nleq p_{k+1}$ because $p_k \leq (p_{k+1})_* < p_{k+1}$ is impossible. Due to figure \ref{kappa} and how the $B$ relation is defined, $x + \kappa(p_k+1) \ngeq p_k$ for all $x < p_{k+1}$ and $y + (p_k)_* \ngeq p_k$ for all $y < p_{k+1}$. This can be more concisely written as $p_{k+1} D p_k$. \\

The minimal pair interpretation is as follows. A join cover, $J_k$, of $p_k$ that would make $\langle p_k ; J_k \rangle$ a minimal pair would be: \\

$J_k = \{p_{k+1} \} \cup \{$ canonical join representation of $\kappa(p_k) \}$ when $p_k A p_{k+1}$ \\

$J_k = \{p_{k+1} \} \cup \{$ canonical join representation of $(p_k)_* \} $ when $p_k B p_{k+1}$. \\

So we have $\langle p_k, J_k \rangle$ representing $p_k C p_{k+1}$. Say that $\langle p_k, J_k \rangle$, with $p_{k+1} \in J_k$, is of the form $(A)$ if it represents $p_k A p_{k+1}$ and similarly if $\langle p_k, J_k \rangle$ is of the form $(B)$. We will call such cycles \emph{special} cycles of minimal pairs. As opposed to presenting $C-$cycles as a collection of special cases (see \cite{FL} or \cite{VL}), this approach has the advantage that one can see where the $A$ and the $B$ components of the $C$ relation come from more clearly. In particular, it seems to emphasize the importance of the bijection $\kappa : J(L) \to M(L)$ and the corresponding duality between $J(L)$ and $M(L)$. \\

In the literature, this relation is sometimes broken into several cases; sometimes depending on the version, of the $C$ relation, used. As evidenced by Nation's proof of J\'onsson's conjecture, this has the advantage of analyzing $C-$cycles at a deeper level. For instance, one can consult \cite{VL} or \cite{FL}. Since our emphasis is different, we will not present these cases. \\

As stated in \cite{FL}, research preceding Nation's proof of J\'onsson's conjecture indicated that if $C-$cycles exist in an $S-$lattice they cannot be too short or too simple. For instance, in \cite{ARJN} J\'onsson and Nation prove that such cycles have at least five elements (or five minimal pairs depending on how one views this)

\begin{theorem}[J\'onsson and Nation. \cite{ARJN}, Theorem 8.1] \label{sampling possibilities}
	Any cycle $p_0 C p_1 C \dots C p_n C p_0$ in an $S-$lattice contains at least two $A$'s and three $B$'s
\end{theorem}

We investigate a bit of this using our new perspective. The below special case of the above theorem will be proven using special minimal pairs. \\

\begin{lemma}[See \cite{ARJN} et.al]
	A finite semidistributive lattice $L$ does not contain any $C-$cycles of length two.
\end{lemma}

\begin{proof}
	This will be done in two cases. Let $p_1 C p_2$ and $p_2 C p_1$.
	
	Case $p_1 A p_2$ and $p_2 C p_1$: If $p_2 A p_1$ then $p_1 < p_2$. And if $p_2 B p_1$ then, by considering the minimal pair associated with $p_2 B p_1$, we see that $p_2 \nleq p_1$. But as $p_1 A p_2$, $p_2 < p_1$ which is impossible by the above.
	
	Case $p_1 B p_2$ and $p_2 B p_1$: Then, since $p_1 \nleq p_2$ and $p_2 \nleq p_1$, $p_1 \parallel p_2$. As $p_1, p_2 \in J(L)$ we have $p_1(p_2)_* = (p_1)_*(p_2) = p_1p_2$. Now, let $\{ i,j \} = \{ 1,2 \}$. Then $p_i + (p_j)_* \leq ((p_i)_* + p_j) + (p_j)_* = (p_i)_* + p_j$, showing that $q = p_1 + (p_2)_* = (p_1)_* + p_2$ for some $q \in L$. But then by $(SD_\vee)$, $(p_1)_* + (p_2)_* = (p_1)_* + (p_2)_* + p_1p_2 = \bigvee \{ ab : a \in \{p_1, (p_2)_* \}, b \in \{(p_1)_*, p_2 \} \} = q$. However, $\{ (p_1)_*, (p_2)_*\}$ is also a proper refinement of both $\{p_1, (p_2)_*\}$ and $\{ (p_1)_*, p_2 \}$; but that is impossible by our minimal pair interpretation.
\end{proof}

We now combine the above with J.B. Nation's duality result, lemma \ref{duality}. To the best of my knowledge the following result I made is not stated elsewhere in the literature. Perhaps some of the results proven by Nation and others can be used here to improve the bound, $n/2$, to $n/a$ for some integer $a > 2$. Below, addition of indices is done modulo $n$: so $J_{k+1} = J_{k'}$ where $1 \leq k' \leq n$ and $k' \cong k+1$ ($\mod$ $n$).

\begin{theorem} \label{fruitful}
	The following are equivalent:
	
	(1) J\'onsson's Conjecture
	
	(2) An S-lattice cannot contain a cycle of minimal pairs $\{ \langle p_k ; J_k \rangle : k = 1,2,\dots,n \}$ with $n \geq 3$ such that $| \{ 1 \leq k \leq n : \{p_k \} \ll J_{k+1} \} | \leq n/2$.
\end{theorem}

\begin{proof}
	
	If $(1)$ is true, assume that $L$ is an $S-$lattice. By (cite references) $L$ is a finite sublattice of a free lattice if and only if it has no cycle of minimal pairs. J\'onsson's conjecture implies that $L$ is a sublattice of a free lattice because $L$ is an $S-$lattice, and so $L$ satisfies $(2)$. \\
	
	Conversely, we will show that the following implies J\'onsson's conjecture: \\
	
	(3) An S-lattice cannot contain a cycle of \emph{special} minimal pairs $\{ \langle p_k ; J_k \rangle : k = 1,2,\dots,n \}$ with $n \geq 3$ such that (addition being done modulo $n$) $| \{ 1 \leq k \leq n : \{p_k \} \ll J_{k+1} \} | \leq n/2$. \\
	
	Assume that $L$ is an $S-$lattice. Then if $L$ has a cycle it must be a cycle of two special minimal pairs, each being of the form $(A)$ or $(B)$. Suppose not, then by J\'onsson and Nation's work there is a special cycle of minimal pairs $C_n = \{ \langle p_i ; J_i \rangle : i = 1,2, \dots, n \}$ with $n \geq 3$. \\
	
	Define an unary relation $U_n$ on $\{p_i : i = 1,2, \dots, n \}$ by $p_i \in U_n$ iff $\{p_i \} \ll J_{i+1}$. Then by $(2)$, we have that $|U_n| > n/2$. If $p_k \in U_n$ and $\langle p_k ; J_k \rangle$ is of the form $(A)$, then as $p_{k+1} < p_k$ and $\{p_k\} \ll J_{k+1}$, $\{p_{k+1}\} \ll J_{k+1}$ which is impossible. \\
	
	So that possibility cannot happen; hence $U_n$ consists of elements $p_i$ where $\langle p_i ; J_i \rangle$ is of the form $(B)$. This is where we use the $\kappa$ function. We first consider the elements: $\kappa(p_1), \kappa(p_2), \dots \kappa(p_n)$. We set $M_i = \{p_{i+1}\} \cup M_i'$ where $M_i'$ is the canonical meet representation of $\kappa(p_i)$ if $\kappa(p_i) A^d \kappa(p_{i+1})$ and $M_i'$ is the canonical meet representation of $(p_i)^*$ if $\kappa(p_i) B^d \kappa(p_{i+1})$. This gives us a cycle of ``dual'' special minimal pairs $C_n^d = \{ \langle \kappa(p_i), M_i \rangle \}$.  Even though the $\kappa$ function is a bijection, it does not even have to be order preserving. An example is the pentagon, $N_5$. This lattice is semidistributive and it satisfies Whitman's condition. However, $\kappa : J(N_5) \to M(N_5)$ is a cyclic permutation on $N_5 \backslash \{0,1\}$; $0$ and $1$ being the bottom and top elements of $N_5$ respectively. So we will use J.B.Nation's duality result, lemma \ref{duality}. \\
	
	We first note that condition $(3)$ is self-dual in the following sense. The join semidistributive laws are the dual of the meet semidistributive laws, and Whitman's condition is a self dual property. So if $(3)$ is true, then condition $(3)$ with \emph{dual special minimal pairs} replacing \emph{special minimals pairs} is also true. \\
	
	So with $C_n^d$, one can define an unary relation $U_n^d$ on $\{\kappa(p_i) : i = 1,2, \dots, n \}$ defined by $\kappa(p_i) \in U_n^d$ iff $M_{i+1} \ll \kappa(p_i)$. By lemma \ref{duality}, $\kappa(p_i) B^d \kappa(p_{i+1})$ if $p_i A p_{i+1}$, and $\kappa(p_i) A^d \kappa(p_{i+1})$ if $p_i B p_{i+1}$. Whence, \\
	
	$|\{\kappa(p_i) : \langle \kappa(p_i), M_i \rangle \text{ is of the form } (A^d)  \}| = |\{p_i : \langle p_i, J_i \rangle \text{ is of the form } (B)  \}| > n/2$ \\
	
	But this violates the dual of condition $(3)$. Hence the only possible cycles of $C-$reduced minimal pairs in $L$ have to consist of two $C-$reduced minimal pairs; but as $L$ is semidistributive that is impossible by a lemma earlier described in these notes.
	
\end{proof}

We note the the above theorem is an improvement over Gaskill, Gr\"atzer, and Platt's observation, see \cite{STL}, that J\'onsson's conjecture is equivalent to asserting that an $S-$lattice does not have a cycle of minimal pairs.

\subsubsection{Some examples:}

To see that there are minimal cycles of minimal pairs, with at least three minimal pairs, that are not as specified in theorem \ref{fruitful} do exist one could consider the following. Below I had constructed three lattices. The most desired example is the right most lattice, with the other diagrams denoting intermediate stages of its construction. It is desired as it is ``closest'' to an $S-$lattice; it is a finite lattice satisfying $(W)$. Starting with the leftmost lattice, Alan Day's interval doubling procedure was used until the rightmost lattice resulted.

\begin{fig}
	\begin{center}
		\begin{tikzpicture}
		\node (1) at (-0.5,2.5) {$\circ$};
		\node (tpr) at (1.25,1.75) {$\circ$};
		\node (tpr1) at (1.33,0.83) {$\circ$};
		\node (tpr2) at (1.42,-0.08) {$\circ$};
		\node (z) at (0.5,1.2) {$z$};
		\node (center) at (0,0) {$\circ$};
		\node (a) at (-0.7,-1) {$a$};
		\node (y) at (-0.35,0.75) {$y$};
		\node (b) at (0.5,-1) {$b$};
		\node (x) at (-1.5,-1) {$x$};
		\node (c) at (1.5,-1) {$c$};
		\node (0) at (0,-2) {$\circ$};
		
		\draw (y) -- (1) -- (x) -- (0) -- (c) -- (tpr2) -- (tpr1) -- (tpr) -- (z) -- (b) -- (0) -- (a) -- (center) -- (tpr1);
		\draw (1) -- (y) -- (a);
		\draw (1) -- (tpr);
		\draw (b) -- (center);
		\draw (a) -- (tpr2);
		\end{tikzpicture}
		\begin{tikzpicture}
		\node (1) at (-0.75,2.5) {$\circ$};
		\node (tpr) at (1.25,1.75) {$\circ$};
		\node (tpr1) at (1.33,0.5) {$\circ$};
		\node (tpr1') at (1.33, 1.2) {$\circ$};
		\node (tpr2) at (1.42,-0.08) {$\circ$};
		\node (z) at (0.5,1.2) {$z$};
		\node (center) at (0,-0.3) {$\circ$};
		\node (center') at (0,0.3) {$\circ$};
		\node (a) at (-0.7,-1) {$a$};
		\node (a') at (-0.7,-0.4) {$\circ$};
		\node (y) at (-0.5,0.75) {$y$};
		\node (b) at (0.5,-1) {$b$};
		\node (x) at (-1.5,-1) {$x$};
		\node (c) at (1.5,-1) {$c$};
		\node (0) at (0,-2) {$\circ$};
		
		\draw (y) -- (1) -- (x) -- (0) -- (c) -- (tpr2) -- (tpr1) -- (tpr1') -- (tpr) -- (z) -- (b) -- (0) -- (a) -- (center) -- (tpr1);
		\draw (1) -- (y) -- (a') -- (a);
		\draw (1) -- (tpr);
		\draw (b) -- (center);
		\draw (a) -- (tpr2);
		\draw (center) -- (center');
		\draw (a') -- (center') -- (tpr1');
		\end{tikzpicture}
		\begin{tikzpicture}
		\node (1) at (-1,2.5) {$\circ$};
		\node (tpr) at (1.25,1.9) {$\circ$};
		\node (tpr1) at (1.33,0.5) {$\circ$};
		\node (tpr1') at (1.33, 1.2) {$\circ$};
		\node (tpr2) at (1.42,-0.3) {$\circ$};
		\node (z) at (0.5,1.2) {$z$};
		\node (center) at (0,-0.5) {$\circ$};
		\node (center'') at (0,-0.1) {$\circ$};
		\node (center') at (0,0.6) {$\circ$};
		\node (a) at (-0.7,-1.2) {$a$};
		\node (a') at (-0.7,-0.2) {$\circ$};
		\node (y) at (-0.7,0.75) {$y$};
		\node (b) at (0.5,-1.2) {$b$};
		\node (x) at (-1.5,-1) {$x$};
		\node (c) at (1.5,-1) {$c$};
		\node (0) at (0,-2) {$\circ$};
		
		\draw (y) -- (1) -- (x) -- (0) -- (c) -- (tpr2) -- (tpr1) -- (tpr1') -- (tpr) -- (z) -- (b) -- (0) -- (a) -- (center) -- (center'') -- (tpr1);
		\draw (1) -- (y) -- (a') -- (a);
		\draw (1) -- (tpr);
		\draw (b) -- (center);
		\draw (a) -- (tpr2);
		\draw (center) -- (center'') -- (center');
		\draw (a') -- (center') -- (tpr1');
		\end{tikzpicture}
	\end{center}
\end{fig}

They each have this cycle of three minimal pairs $\langle a ; \{b,c\}  \rangle$, $\langle b ; \{c,a\}  \rangle$, and $\langle c ; \{a,b\}  \rangle$ such that $\{a\} \ll \{y,c\}$, $\{b\} \ll \{z,a\}$, and $\{c\} \nless \nless \{x,b\}$. This example also has the property that $\{a\} \cap \{y,c\} = \{b\} \cap \{z,a\} = \varnothing$.

\section{Some Properties of Finite Semidistributive Lattices}

We look at properties that certain finite semidistributive lattices possess and do two things. See how they can be used to extend this different way of approaching J\'onsson's conjecture. And see limitations of the semidistributive laws in analyzing finite lattices in general. In the process, Rival and Sand's work related to J\'onsson's conjecture is presented.

\subsection{Breadth Two Semidistributive Lattices and Related Properties}

Let $L$ be a lattice. The \emph{breadth}, $b(L)$, of a lattice $L$ is (if it exists) the smallest number $N$ such that the following is true for all $x \in L$: if $S$ is a join representation of $x$, then for some $S' \subseteq S$, with $|S'| = N$, $S'$ is also a join representation of $x$. This definition of breadth is self dual in the sense that the above definition would not change if \textit{join representation} were replaced with \textit{meet representation}. A proof that there is no difference can be found in Rival and Sands' paper \cite{PSL}. \\

We consider partial orders with elements, $\{x_1, x_2, x_3, x_4, \dots x_{2n-1},x_{2n}\}$, known as \emph{crowns of order $2n$} (see \cite{PSL}). Alternatively, they are called \emph{$n-$crowns} (see \cite{HM}). A partial order as specified above is a crown of order $2n$ if its only comparability relations are $x_{2j-1} < x_{2j}$ and $x_{2j} > x_{2j+1}$ for $j = 1, 2, \dots, n$. Arithmetic is done modulo $n$. A crown is defined to be an $n-$crown for some positive integer $n$. See \cite{PSL}, \cite{PSL2}, and \cite{HM} for more on crowns. \\

A finite lattice $L$ is said to be \emph{dismantlable} if there exists a labelling of its elements, $\{a_1, a_2, \dots, a_n\}$, such that $\{a_1\}$, $\{a_1,a_2\}$, $\dots$, $\{a_1,a_2, \dots, a_{n-1}\}$ are all sublattices of $L$. So, see \cite{HM}, with such lattices one can sequentially delete doubly irreducible elements until only one element remains; hence the term. We present the following lemmas from \cite{PSL}: \\

\begin{lemma}[\cite{PSL}, Lemma 2.6] \label{see}
	\label{1}
	\textit{Let $L$ be a finite lattice satisfying $(SD_\wedge)$ and let $n^*$ be the maximum number of elements covering any element of $L$. Then $n^* = b(L)$.}
\end{lemma}

The above lemma says that for finite \emph{semidistributive} lattices, one can ``see'' the breadth of the structure by looking at its covering relation. Because the definition of breadth is self-dual, we see that a finite semidistributive lattice $L$ has breadth $N = b(L)$ if and only if any element $x$ in $L$ covers at most $N$ elements and is covered by at most $N$ elements. The following characterization of dismantlable lattices was proven by D. Kelley and I. Rival in 1974. \\

\begin{lemma}[D. Kelly and I. Rival 1974. See \cite{PSL}, Lemma 2.3]
	\label{2}
	\textit{A finite lattice is dismantlable if and only if it contains no crown.}
\end{lemma}

We also have the following characterization:

\begin{lemma}[\cite{PSL}, Lemma 2.2]
	\label{3}
	\textit{A lattice has breadth at most two if and only if it contains no crown of order six.}
\end{lemma}

Combining the above two lemmas, Rival and Sands proved the following characterization in \cite{PSL} \\

\begin{lemma}[\cite{PSL}, Lemma 2.4]
	\label{4}
	\textit{Let $L$ be a finite lattice satisfying $(SD_\vee)$. Then $L$ is dismantlable if and only if $b(L) \leq 2$.}
\end{lemma}

A similar result to the above lemma, namely the breadth of a finite semidistributive dismantlable lattice does not exceed two, was obtained in H. M\"uhle's manuscript \cite{HM}. However, this manuscript has a serious fallacy, as explained in the last subsection: \emph{Some Properties of Finite Semidistributive Lattices}. \\

So by Lemma \ref{4} and Lemma \ref{1} a semidistributive lattice $L$ is dismantlable if and only if every element of $L$ covers at most two other elements and is covered by at most two other elements. We conclude, from the above, that a finite semidistributive lattice \emph{does not have a crown if and only if it does not have a crown of order $6$}. \\

\subsection{J\'onsson's conjecture and crowns of order 2n}

A nice feature of cycles of minimal pairs is that, although they are more blunt than $C-$cycles, they appear to be quite versatile. Consider the following cycle of minimal pairs: $\langle x_1, A_1 \rangle$, $\langle x_2, A_2 \rangle$, $\dots$, $\langle x_n, A_n \rangle$ with $n \geq 3$. We add two further assumptions to these cycles, \\

(1) $\{ \bigvee A_1, \bigvee A_2, \dots, \bigvee A_n \}$ is an antichain \\

(2) $x_i \nless \nless A_j$ if $j \neq i$ and $x_i \notin A_j$ \\

For brevity, we will call such cycles of minimal pairs \emph{crowned cycles}, for the following reason: the sub partial order of the lattice with elements $x_1, \bigvee A_1, x_2, \bigvee A_2, \dots, x_n, \bigvee A_n$ is a crown of order $2n$. Comparing with theorem \ref{fruitful}, we note that such cycles satisfy $|\{i = 1,2, \dots, n : \{x_i\} \ll A_i \}| = 0 < n/2$. So if a finite semidistributive lattice has a crowned cycle, the ideal would be to show that it also fails $(W)$. What we do have is that if there was a crowned cycle there would also have to be, by lemmas \ref{2}, \ref{3}, and \ref{4}, one consisting of three minimal pairs. \\

By theorem \ref{sampling possibilities}, all \emph{special} cycles of minimal pairs have to consist of at least five elements. So if that could be adapted to \emph{crowned} cycles of minimal pairs, it would imply that an $S-$lattice cannot have a crowned cycle of minimal pairs. Such a theorem would extends theorem \ref{fruitful} since crowned cycles were not excluded there. I will make the following example, the purpose of which is to show that crowned cycles in finite lattices do exist in general. \\
\begin{fig}
	\begin{center}
		\begin{tikzpicture}
		\node (c') at (1,0) {$c'$};
		\node (c) at (2,0) {$c$};
		\node (b) at (0,0) {$b$};
		\node (a') at (3,0) {$a'$};
		\node (b') at (-1,0) {$b'$};
		\node (a) at (-3,0) {$a$};
		
		\node (b+b') at (-2,1.7) {$b+b'$};
		\node (c+c') at (0,1.7) {$c+c'$};
		\node (a+a') at (2,1.7) {$a+a'$};
		
		\node (bb') at (-2,-1.7) {$\circ$};
		\node (cc') at (0,-1.7) {$\circ$};
		\node (aa') at (2,-1.7) {$\circ$};
		
		\node (0) at (0,-2.6) {$\circ$};
		\node (1) at (0,2.6) {$\circ$};
		
		\draw (1) -- (a+a');
		\draw (1) -- (b+b');
		\draw (1) -- (c+c');
		
		\draw (0) -- (aa');
		\draw (0) -- (bb');
		\draw (0) -- (cc');
		
		\draw(a+a') -- (c);
		\draw(a+a') -- (a');
		\draw(a+a') -- (a);
		
		\draw(b+b') -- (a);
		\draw(b+b') -- (b');
		\draw(b+b') -- (b);
		
		\draw(c+c') -- (b);
		\draw(c+c') -- (c');
		\draw(c+c') -- (c);
		
		\draw(aa') -- (a);
		\draw(aa') -- (a');
		
		\draw(bb') -- (b);
		\draw(bb') -- (b');
		
		\draw(cc') -- (c);
		\draw(cc') -- (c');
		\end{tikzpicture}
	\end{center}
\end{fig}
The crowned cycle in this structure is $\langle a ; \{b',b\} \rangle$, $\langle b ; \{c', c\} \rangle$, $\langle c ; \{a', a\} \rangle$. A nice feature of this lattice is that it also satisfies $(W)$. We therefore note that it cannot be semidistributive. If this lattice is semidistributive, J\'onsson's conjecture would be incorrect.

\section{Generalization to Finite Semidistributive Lattices}

Assume that $L$ is a finite semidistributive lattice. We see how our analysis would change if the requirement that $L$ is planar is assumed instead of Whitman's condition. This emphasises the difficulty of analysing finite semidistributive lattices in general. A lattice $L$ is planar if its H\"asse diagram can be drawn (on the plane) in such a way that no two lines drawn intersect each other. 

\subsection{The Snakes}

A manuscript by Henri M\"uhle entitled \textit{Finite dismantlable semidistributive lattices are planar}, \cite{HM}, was submitted in 2015 to arxiv. Henri M\"uhle had defended his PhD thesis on partial orders associated with reflection groups, at Universit\"at Wien in 2014, and is associated with LIAFA, Universit\'e Paris Diderot. The goal of the manuscript was to show that all finite dismantlable semidistributive lattices are planar. He approached the problem by using saturated crossing paths and a 1975 forbidden sublattice characterization of semidistributive lattices. There appear to be, however, examples from Ivan Rival and Bill Sands' paper \cite{PSL2} of finite dismantlable semidistributive lattices that are not planar. These are counterexamples to Lemma 3.5 of \cite{HM} where certain semidistributive lattices were considered. I had noted this and emailed H. M\"uhle about the error. \\

The counterexamples turn out to be finite dismantlable semidistributive lattices that are not planar, and are informally known as \emph{snakes} by Rival and Sands. They were denoted $S_n$ in \cite{PSL2}. The lattices $S_n$ were the new lattices introduced by Rival and Sands in their paper \cite{PSL2}, which will be described shortly. The authors unofficially called such structures snakes. The notation $S_n$ comes from \cite{PSL2}. \\

In \cite{PSL2}, Ivan Rival and Bill Sands proved a characterization of planar sublattices of free lattices. So as J\'onsson's conjecture is true, they had, in particular, proven that \\

\begin{theorem}[I.Rival and B.Sands: \cite{PSL2}, see Theorem 1.6]
	A finite sublattice of a free lattice is planar if and only if it does not contain the eight element boolean algebra $2 \times 2 \times 2$, or any snake $S_n$ as a sublattice.
\end{theorem}

A consequence is that,

\begin{theorem}[I. Rival and B.Sands]
	A finite sublattice of a free lattice of breadth two is planar if and only if it does not contain any snake $S_n$ as a sublattice.
\end{theorem}

Two of the snakes are as follows. The labels $a_i$ and $b_i$, and the label $c$, for the $S_n$ also come from \cite{PSL2}. \\

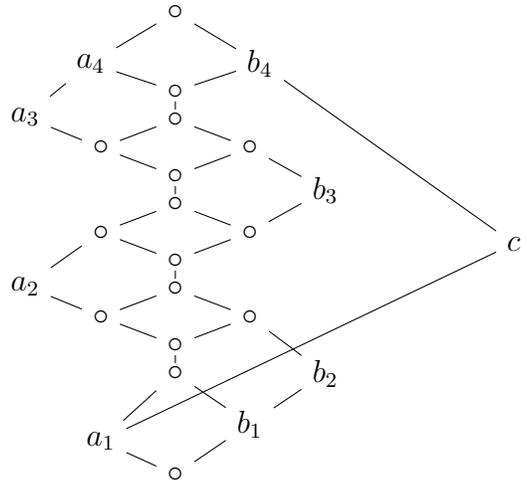
\begin{fig} \label{snakes}
	\begin{center}
		\begin{tikzpicture} [scale=1.5]
		\node (1) at (0,0) {$1$};
		\node (a3) at (-0.5,-0.5) {$a_3$};
		\node (b3) at (0.5,-0.5) {$b_3$};
		\node (A) at (0,-1) {$A$};
		\node (B) at (0,-1.5) {$B$};
		\node (a2) at (-1,-1.25) {$a_2$};
		\node (C) at (-0.5,-2) {$C$};
		\node (D) at (0.5,-2) {$D$};
		\node (E) at (0, -2.5) {$E$};
		\node (F) at (0,-3) {$F$};
		\node (b2) at (1,-2.75) {$b_2$};
		\node (b1) at (0.5,-3.5) {$b_1$};
		\node (a1) at (-0.5,-3.5) {$a_1$};
		\node (0) at (0, -4) {$0$};
		
		\node (c) at (2.5,-2.5) {$c$};
		
		\draw (1) -- (a3) -- (A) -- (b3) -- (1);
		\draw (a3) -- (a2) -- (C) -- (B) -- (A);
		\draw (B) -- (D) -- (E) -- (C);
		\draw (D) -- (b2) -- (b1) -- (F) -- (E);
		\draw (b1) -- (0) -- (a1) -- (F);
		
		\draw (b3) -- (c) -- (a1);
		\end{tikzpicture}
		\begin{tikzpicture} [scale=1.5]
		\node (1) at (0,0.2) {$\circ$};
		\node (a4) at (-0.75,-0.25) {$a_4$};
		\node (b4) at (0.75,-0.25) {$b_4$};
		\node (A) at (0,-0.5) {$\circ$};
		\node (B) at (0,-0.75) {$\circ$};
		\node (a3) at (-1.33,-0.725) {$a_3$};
		\node (C) at (-0.66,-1) {$\circ$};
		\node (D) at (0.66,-1) {$\circ$};
		\node (E) at (0, -1.25) {$\circ$};
		\node (F) at (0,-1.5) {$\circ$};
		\node (b3) at (1.33,-1.375) {$b_3$};
		\node (G) at (-0.66,-1.75) {$\circ$};
		\node (H) at (0.66,-1.75) {$\circ$};
		\node (I) at (0, -2) {$\circ$};
		\node (J) at (0, -2.25) {$\circ$};
		\node (a2) at (-1.33,-2.225) {$a_2$};
		\node (K) at (-0.66, -2.5) {$\circ$};
		\node (L) at (0.66, -2.5) {$\circ$};
		\node (M) at (0, -2.75) {$\circ$};
		\node (N) at (0, -3) {$\circ$};
		\node (b2) at (1.33, -3) {$b_2$};
		\node (b1) at (0.66, -3.45) {$b_1$};
		\node (a1) at (-0.66, -3.6) {$a_1$};
		\node (0) at (0, -3.9) {$\circ$};
		
		\node (c) at (3,-1.85) {$c$};
		
		\draw (1) -- (a4) -- (A) -- (b4) -- (1);
		\draw (a4) -- (a3) -- (C) -- (B) -- (A);
		\draw (B) -- (D) -- (E) -- (C);
		\draw (D) -- (b3) -- (H) -- (F) -- (E);
		\draw (H) -- (I) -- (G) -- (F);
		\draw (G) -- (a2) -- (K) -- (J) -- (I);
		\draw (J) -- (L) -- (M) -- (K);
		\draw (M) -- (N) -- (b1) -- (b2) -- (L);
		\draw (b1) -- (0) -- (a1) -- (N);
		
		\draw (b4) -- (c) -- (a1);
		\end{tikzpicture}
	\end{center}
\end{fig}

The left lattice is $S_0$, having a single ``S'' pattern for a ``body'' whose ``head'' and ``tail'' are linked with $c$ as shown above. The right lattice is $S_1$, whose ``body'' consists of two ``S'' patterns linked together. The top square of the first ``S'' is glued to the bottom square of the second ``S''. Like with $S_0$, its ``head'' and ``tail'' are also linked with $c$ as shown above. \\

In general the ``body'' of $S_n$ has $n+1$ ``S'' patterns linked together. The top square of the first ``S'' is glued to the bottom square of the second ``S'', the top square of the second ``S'' is glued to the bottom square of the third ``S'', and so on until the top square of the $n$th ``S'' is glued to the bottom square of the $n+1$th ``S''. Finally, the extra element $c$ is linked with the ``head'' and ``tail'' of $S_n$, just like it was with $S_0$ and $S_1$. \\

In the next subsection, I will prove that these snakes are indeed sublattices of free lattice by explicitly defining an embedding $S_n \to FL(X)$.

\subsection{Snakes are sublattices of free lattices}

I will prove the following theorem by explicitly constructing an embedding $S_n \to FL(X)$. There are more straightforward ways of proving this particular theorem. However, the below construction is useful as it hints at the difficulty of embedding lattices into free lattices. \\

In \cite{PSL2}, Rival and Sands establish a \emph{forbidden sublattice} characterization of \emph{finite planar sublattices of free lattices}. The construction I made in the below proof is connected with the proof by I.Rival and B.Sands of one of the major theorems in \cite{PSL2}. It should be noted that a similar argument to the one below, likely different in some details, was probably made by Rival and Sands in \cite{PSL2}. The below should be seen as an introduction to some of the concepts and the content in \cite{PSL2}. \\

\begin{theorem}[Also, see \cite{PSL2} Theorem 1.2 and Lemma 4.1] \label{embedding snakes}
	The snakes $S_n$ are semidistributive and satisfy Whitman's condition.
\end{theorem}

\begin{proof} 
	
	For the labellings of elements we refer to diagram of $S_0$ in figure \ref{snakes} in the previous subsection. I will prove the special case when $S_n = S_0$. The proof can then be slightly extended to give a proof of the theorem. Below, the elements of $S_0$ will be presented in boldface: \textbf{1}, \textbf{a$_3$}, \textbf{b$_3$}, \textbf{A}, \textbf{B}, \textbf{a$_2$}, \textbf{C}, \textbf{D}, $\dots$, \textbf{b$_1$}, \textbf{0}. The symbols come from the diagram of $S_0$, shown above. \\
	
	To show that the $S_0$ is semidistributive and satisfies Whitman's condition, a sublattice of a free lattice isomorphic to $S_0$ will be constructed. In constructing this, the partial orders involved become very similar to the partial orders used by Rival and Sands in \cite{PSL2} which they label $G_n$ and $R_n$. We consider the free lattice on $\omega$ generators, $FL(\omega)$, with generating set $H = \{a,b,c, \dots \}$. None of these symbols will be boldfaced. In particular $c$, \textbf{c}, and \textbf{C} are all different since the first is a generator in $H$ and the other two are elements of $S_0$. \\
	
	Two ideas will be used, the first is that the lattice $ \{ a, a + bc, a(b + c), (a + bc)(b + c), a(b + c) + bc \}$ is isomorphic to the pentagon, $N_5$. The second is the idea of inserting generators into intervals; an application of this can be found in Corollary 9.11 of \cite{FL}. The corollary is a simple extension of Tschantz's theorem, see \cite{FL}. The insertion idea is as follows. Let $w$ be a generator, $w \in H$, and $[\alpha, \beta]$ an interval in $FL(\omega)$. Moreover, assume that $\alpha$ and $\beta$ can be expressed using generators from $H$ that are not $w$. Then we take $w' = \alpha + w\beta$, or we take $w' = (\alpha + w)\beta$. In both cases we have $w' \in [\alpha, \beta]$. \\
	
	Firstly, start off with the sublattice $N_5 = \{ a, a + bc, a(b + c), (a + bc)(b + c), a(b + c) + bc \}$ and add two extra elements $bc$ and $b + c$. We now take two generators from $H \backslash \{a,b,c\}$, $x$ and $y$. Insert $x$ and $y$ into $[a(b+c) + bc, b + c]$ and $[bc, (a + bc)(b + c)]$ respectively to form \textbf{a$_2$} $= (a(b+c) + bc + x)(b + c)$ and \textbf{b$_2$} $= (a(b+c) + bc + x)(b + c)$. The set $N_5 \cup \{ b+c, bc, (a(b+c) + bc + x)(b + c), (a(b+c) + bc + x)(b + c) \}$ now forms a partial order that resembles $G_n$ or $R_n$ as described in \cite{PSL2}. \\
	
	Now take the meet, \textbf{C}, of \textbf{a$_2$} with $(a + bc)(b + c)$ and the join, \textbf{D}, of \textbf{b$_2$} with $a(b + c) + bc$. Due to free lattices satisfying Whitman's condition we can set \textbf{B} = \textbf{C} $+$ \textbf{b$_2$}, and \textbf{E} = \textbf{a$_2$}\textbf{D}. After doing all this, we now set \textbf{a$_3$} = \textbf{a$_2$} $+$ \textbf{b$_2$} and \textbf{b$_1$} = \textbf{a$_2$}\textbf{b$_2$}. Continuing with this process and then simplifying the expressions, we obtain: \\
	
	\textbf{1} $= a + (a(b+c) + bc + x)(b + c)$ \\
	
	\textbf{b$_3$} $= a + bc$ \\
	
	\textbf{a$_3$} $= (a(b+c) + bc + x)(b + c) + (a + bc)(b + c)y$ \\
	
	\textbf{a$_2$} $= (a(b+c) + bc + x)(b + c)$ \\
	
	\textbf{A} $= (a + bc)((a(b+c) + bc + x)(b+c) + (a+bc)(b+c)y)$ \\
	
	\textbf{B} $= (a(b+c) + bc + x)(a + bc)(b + c) + (a + bc)(b + c)y$ \\
	
	\textbf{C} $= (a(b+c) + bc + x)(a + bc)(b + c)$ \\
	
	\textbf{D} $= (a + bc)(b+c)y + a(b + c) + bc$ \\
	
	\textbf{E} $= ((a + bc)(b+c)y + a(b+c) + bc)(a(b+c) + bc + x)$ \\
	
	\textbf{F} $= a(b+c) + ((a + bc)(b+c)y + bc)(a(b+c) + bc + x)$ \\
	
	\textbf{b$_2$} $= (a+bc)(b+c)y + bc$ \\
	
	\textbf{b$_1$} $= ((a + bc)(b + c)y + bc)(a(b+c) + bc + x)$ \\
	
	\textbf{a$_1$} $= a(b + c)$ \\
	
	\textbf{0} $= a((a + bc)(b + c)y + bc)$ \\
	
	\textbf{c} $= a$ \\
	
	To see how the proof of this special case $S_n = S_0$ can be transferred to the more general case, the following can be said: That these structures are dismantlable and non-planar follows from the arguments just given. To construct sublattices of $FL(\omega)$ isomorphic to the $S_n$ when $n \geq 1$, we replace $x$ and $y$ in the above argument with $x+y$ and $xy$. Then continue to insert more elements of the generating set of $FL(\omega)$ in such a way that partial orders resembling the $G_n$ and $R_n$ from \cite{PSL2} are made with added top and bottom elements. How the construction proceeds is similar to the above proof.
	
\end{proof}

\subsection{Extending Rival and Sands' paper to finite planar semidistributive lattices}

We look at the limitations of the semidistributive laws on finite lattices. Specifically, we narrow our attention to finite planar lattices. From J\'onsson's conjecture, one may guess that semidistributivity is a very strong condition for finite lattices. We describe a limitation of this point of view by describing the difficulty in obtaining a similar characterization for finite planar semidistributive lattices. In the process we look at Rival and Sands characterization of finite sublattices of free lattices that are also planar. And as this is done, an error in Henri Muhle's manuscript (which would have contributed to this characterization problem) is pointed out. Determining a characterization for planar semidistributive lattices appears to be a much more daunting task. Inspired by the approaches used in \cite{HM}, we describe a potential direction towards the breadth two / dismantlable case of this problem. \\

An option one could consider is a forbidden sublattice characterization of \emph{breadth-two} planar semidistributive lattices. Such a set, $\mathcal{F}$, of forbidden sublattices should be minimal. We will assume that they satisfy the following (1) Each member of $\mathcal{F}$ is a finite dismantlable semidistributive non-planar lattice, (2) A finite dismantlable semidistributive lattice is planar if and only if it does not contain a member of $\mathcal{F}$ as a sublattice. \\

Below, proposition \ref{sparse}, is a partial result that I made and proved. The approach from \cite{HM} that I have adopted is the idea of, given a finite lattice with top element $1$ and bottom element $0$, using the elements that cover $0$ and using the elements that are covered by $1$. \\

\begin{proposition} \label{sparse}
	There is an $\mathcal{F}$, as described above, such that for all $F \in \mathcal{F}$ the following holds. \\
	
	Let $0$ be its bottom and $1$ be its top element of $F$. Then in $F$, $1$ covers two distinct elements $a$ and $b$, $0$ is covered by two distinct elements $c$ and $c$, and $c + d < ab$. Furthermore, one can find $a',b',c',d' \in F$ with the following properties. \\
	
	(1) $a' \neq ab$ is covered by $a$, $b' \neq ab$ is covered by $b$, $c' \neq c + d$ covers $c$, and $c' \neq c + d$ covers $d$. \\
	
	(2) $|\{a',b',c',d'\}| = 4$, or $|\{a',b',c',d'\}| = 3$ with $a' \neq b'$ and $c' \neq d'$.
\end{proposition}

\begin{proof}
We first show that there is an $\mathcal{F}$ as specified above such that, for all $F \in \mathcal{F}$, the top element of $F$ covers (exactly) two elements $a,b$, the bottom element of $F$ covers (exactly) two elements $c,d$ and $c + d < ab$. This claim will be shown by induction on the number of elements of a finite breadth-two semidistributive lattice $L$. \\

If $|L| \leq 4$ then the claim is trivial. So assume that this is true for some $n \geq 4$ when $|L| \leq n$. And let $L$ be as specified above with $|L| = n + 1$. Because $L$ is a finite semidistributive lattice, that is breadth is two is the same as requiring that every element covers or is covered by at most two elements (see \cite{PSL}). Let $0$ and $1$ be the bottom and top element of $L$ respectively, and assume without loss of generality that $1$ covers exactly two elements $a$ and $b$ and $0$ is covered by exactly two elements $c$ and $d$. If $\{a,b\} \cap \{c,d\} \neq \varnothing$ then $L \cong 2 \times 2$ (the four element boolean algebra/lattice) or $L \cong N_5$ (the pentagon) and in these cases the proposition is clear (the case involving $N_5$ follows by strong induction). So assume that $\{a,b\} \cap \{c,d\} = \varnothing$. Four cases will be looked at. \\

\textbf{Case $ab \parallel c + d$ :} Without loss of generality assume that $c + d \leq a$ and $ab \geq d$. Then we this the following figure possibly with $c+d = a$ or $d = ab$ resulting in four subcases. \\

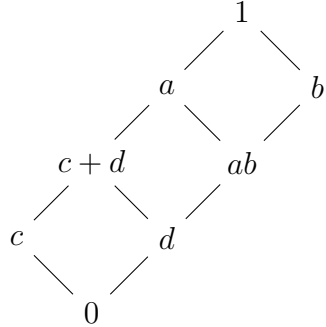
\begin{center}
\begin{tikzpicture}

\node (0) at (0,0) {$0$};
\node (c) at (-1,1) {$c$};
\node (d) at (1,1) {$d$};
\node (c+d) at (0,2) {$c+d$};
\node (ab) at (2,2) {$ab$};
\node (a) at (1,3) {$a$};
\node (b) at (3,3) {$b$};
\node (1) at (2,4) {$1$};

\draw (0) -- (c) -- (c+d) -- (a) -- (1) -- (b) -- (ab) -- (d) -- (0);
\draw (d) -- (c+d);
\draw (a) -- (ab);

\end{tikzpicture}
\end{center}

In any such case the sublattice generated by $\{a,b,c,d\}$ is isomorphic to the lattice obtained from $2 \times n$ (where $n = 2,3,4$) by adding elements $x$ such that $(i,k) < x < (i,k+1)$, for some $i = 1,2$ and $k = 1, \dots, n-2$ , and $2 \times n \cup \{x\}$ is a sublattice of $L$. \\

Denote the sublattice generated by $\{a,b,c,d\}$ by $\langle a,b,c,d \rangle$. Since the structure is a lattice it is impossible to find an $x \in L$ satisfying $(1,k) < (0,l)$ for some $k < l$. Since $L$ is semidistributive, $1$ covers $b$, and $0$ is covered by $c$, it is impossible to find a $y \in L$ satisfying $(0,k) < y < (1,l)$ for some $0 \leq k \leq l \leq n - 1$ (consider $\{c,y,b\}$, this triple would violate join or meet semidistributivity). \\

So now we can say the following, because the $(0,k)$ are meet reducible and the $(1,k)$ are join reducible. By the above, $L$ is not planar if and only if the sublattice $M_{k,l}^0$ of $L$ generated by $[(0,k), (0,l)] \cup \{ (0,k+1), (0,k+2), \dots, (0,l) \}$ or $M_{k,l}^1$, the dual of $M_{k,l}^0$, is not planar for some $k$ and $l$. Now, $\bigwedge M_{k,l}^0$ is meet irreducible; for otherwise $(0,k)$ would be covered by at least three elements contrary to assumption. Dually, $\bigvee M_{k,l}^1$ is join irreducible. \\

So if $L$ is not planar assume without loss of generality that $M_{k,l}^0$ is not planar for some $k$ and $l$. Then $N_{k,l}^0$, the sublattice of $M_{k,l}^0$ obtained by deleting $\bigwedge M_{k,l}^0$, is not planar and $|N_{k,l}^0| < |L|$ so by induction $N_{k,l}^0$ contains a sublattice isomorphic to a lattice as specified in the claim. The converse is clear as all members of $\mathcal{F}$ are assumed to be non-planar. \\

\textbf{Case $ab < c + d$ :} Then without loss of generality assume that $d \leq ab < c + d \leq a$. We obtain this figure possibly with $d = ab$ or $c + d = a$.

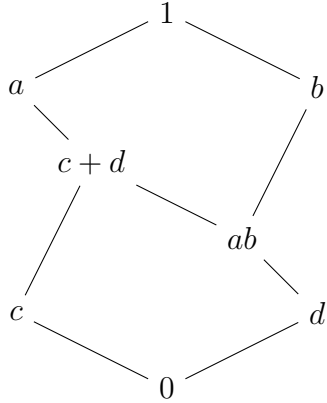
\begin{center}
	\begin{tikzpicture}
	
	\node (0) at (0,0) {$0$};
	\node (c) at (-2,1) {$c$};
	\node (d) at (2,1) {$d$};
	\node (c+d) at (-1,3) {$c+d$};
	\node (ab) at (1,2) {$ab$};
	\node (a) at (-2,4) {$a$};
	\node (b) at (2,4) {$b$};
	\node (1) at (0,5) {$1$};
	
	\draw (0) -- (c) -- (c+d) -- (a) -- (1) -- (b) -- (ab) -- (d) -- (0);
	\draw (c+d) -- (ab);
	
	\end{tikzpicture}
\end{center}

If $ab < x < c+d$ for some $x \in L$ then $x+b = 1$ and $cx = 0$. But then $L$ is neither join nor meet semidistributive with $\{c,x,b\}$. So we have these covering relations: $a \prec 1$, $b \prec 1$, $0 \prec c$, $0 \prec d$, and $ab \prec c+d$. Moreover, it is impossible to find a $y \in L$ such that $d < y < a$, $y \parallel ab$, and $y \parallel c+d$ because otherwise $\{c,y,b\}$ would violate join and meet semidisributivity. \\

And since $L$ is a lattice, it is impossible for $c < b$. Now we use strong induction. If $L$ is non-planar then $L_d$, the sublattice generated by $[d,b] \cup \{c+d, a, 1\}$ or $L_a$, the sublattice generated by $[a,c] \cup \{ab, d, 0\}$, is not planar. \\

Without loss of generality assume that $L_d$ is not planar. then $d = \bigwedge L_d$, hence $|L_d| < |L|$ and by induction $L_d$ contains a sublattice isomorphic to a member of $\mathcal{F}$ as specified in the claim. \\

\textbf{Case $ab = c+d$ :} Let $e = ab = c+d$. Without loss of generality suppose that there is a three element antichain $\{x,e,y\}$ such that $c < x < b$ and $d < y < b$. Since the breadth of $L$ is two, lemma \ref{see} implies that every element covers, or is covered by, at most two elements. Hence, $xy = cd = 0$. But then $a + x = a + y = 1$, but $a + xy = a$ violating the join semidistributive laws. \\

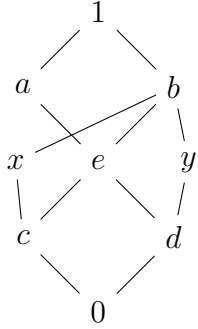
\begin{center}
	\begin{tikzpicture}
	
	\node (0) at (0,0) {$0$};
	\node (c) at (-1,1) {$c$};
	\node (d) at (1,1) {$d$};
    \node (x) at (-1.1,2) {$x$};
    \node (e) at (0,2) {$e$};
    \node (y) at (1.2,2) {$y$};
	\node (a) at (-1,3) {$a$};
	\node (b) at (1,3) {$b$};
	\node (1) at (0,4) {$1$};
	
    \draw (0) -- (c) -- (e) -- (d) -- (0);
    \draw (1) -- (a) -- (e) -- (b) -- (1);
    \draw (c) -- (x) -- (b);
    \draw (d) -- (y) -- (b);

	\end{tikzpicture}
\end{center}

So such antichains, $\{x,e,y\}$, are impossible to find in $L$. Hence induction can be used. If $L$ is not planar then $L_d$, the sublattice generated by $[d,b] \cup \{a,1\}$, or three similarly defined sublattices ($L_c$, $L_b$, and $L_a$), is not planar. Without loss of generality assume that $L_d$ is not planar. then as $|L_d| < |L|$, $L_d$ has a sublattice isomorphic to a member of $\mathcal{F}$ as specified in the claim. \\

This proves the claim, in particular $\mathcal{F}$ can be defined so that all of its members have the following as a sublattice (with $0$ being the bottom element, $1$ being the top element, $a \prec 1$, $b \prec 1$, $0 \prec c$, $0 \prec d$, and $c+d < ab$):

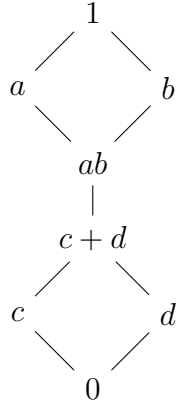
\begin{center}
	\begin{tikzpicture}
	
	\node (0) at (0,0) {$0$};
	\node (c) at (-1,1) {$c$};
	\node (d) at (1,1) {$d$};
	\node (c+d) at (0,2) {$c+d$};
	\node (ab) at (0,3) {$ab$};
	\node (a) at (-1,4) {$a$};
	\node (b) at (1,4) {$b$};
	\node (1) at (0,5) {$1$};
	
	\draw (1) -- (a) -- (ab) -- (b) -- (1);
    \draw (ab) -- (c+d);
    \draw (0) -- (d) -- (c+d) -- (c) -- (0); 
	
	\end{tikzpicture}
\end{center}

When Whitman's condition is assumed, the members of $\mathcal{F}$ are (as shown by Rival and Sands in \cite{PSL2}) the snakes $S_n$; all of whom have the above sublattice as specified above. 

When Whitman's condition is not assumed, and let $F \in \mathcal{F}$. The following can be said: If $a$ is join irreducible, then at least one of the following proper sublattices is non-planar: $[ab,b]$, $[0,b]$. So as the members of $\mathcal{F}$ are assumed to be minimal with respect to set inclusion, $a$ and $b$ are join reducible and $c$ and $d$ are meet reducible. This proves $(1)$; let $a'$, $b'$, $c'$, and $d'$ be as specified in this proposition. \\

If $a' = b'$ then $a'b' \leq ab$ and, as the breadth two condition implies by \ref{see} that every element of $F$ covers or is covered by at most two elements, we have this decomposition $F = [0, ab] \cup \{ab, a, b, a+b = 1\}$ contrary to the assumption that the members of $\mathcal{F}$ are minimal with respect to set inclusion. So $a' \neq b'$ and by symmetry $c' \neq d'$. \\

Now suppose that $a' = c' = x$ and $b' = d' = y$. Then considering the Hasse diagram of the lattice and ``stretching'' $x$ and $y$ apart from each other (until the line segments $xa$, $cx$, $yb$, $dy$ no longer intersect any of the other line segments of the Hasse diagram) shows that the lattice $F$ is not minimal with respect to set inclusion as one can consider the proper (non-planar) sublattice $F \backslash \{x,y\}$. This proves $(2)$.

\end{proof} \\

An interesting feature of this proposition is that such forbidden sublattices appear to resemble the snakes from Rival and Sands described previously.

\end{document}